\documentclass[11pt,bezier]{article}
\usepackage{amsmath, amssymb, amsfonts, graphicx}
\usepackage{pgf,tikz}\usepackage{mathrsfs}

\newtheorem{prethm}{{\bf Theorem}}

\newenvironment{thm}{\begin{prethm}\sl{\hspace{-0.5
               em}{\bf.}}}{\end{prethm}}

\newtheorem{prepro}[prethm]{{\bf Proposition}}

\newtheorem{prelem}[prethm]{{\bf Lemma}}

\newenvironment{lem}{\begin{prelem}\sl{\hspace{-0.5
               em}{\bf.}}}{\end{prelem}}

\newtheorem{predeff}[prethm]{{\bf Definition}}

\newtheorem{precor}[prethm]{{\bf Corollary}}

\newenvironment{cor}{\begin{precor}\sl{\hspace{-0.5
               em}{\bf.}}}{\end{precor}}

\newtheorem{preconj}[prethm]{{\bf Conjecture}}

\newtheorem{preremark}[prethm]{{\bf Remark}}

\newenvironment{remark}{\begin{preremark}\rm{\hspace{-0.5
               em}{\bf.}}}{\end{preremark}}

\newtheorem{preexample}[prethm]{{\bf Example}}

\newtheorem{preproof}{{\bf\textsf{Proof.}}}

\newenvironment{proof}[1]{\begin{preproof}{\rm
               #1}\hfill{$\Box$}}{\end{preproof}}

\newcommand{\la}{\lambda}
\newcommand{\x}{{\bf x}}

\newcommand{\mul}{{\rm mult}}

\newcommand{\f}{{\cal F}}
\newcommand{\cp}{{\rm CP}}
\newcommand{\sg}{\leqslant}

\title{Cographs: Eigenvalues and Dilworth Number}

\author{{\sc Ebrahim Ghorbani} \\[.3cm]
{\sl Department of Mathematics, K. N. Toosi University of Technology,}\\
{\sl P. O. Box 16315-1618, Tehran, Iran}\\
{\sl School of Mathematics, Institute for Research in Fundamental
Sciences (IPM),}\\
{\sl P. O. Box 19395-5746, Tehran, Iran}
\\[.3cm]
$\mathsf{e\_ghorbani@ipm.ir}$}

\begin{document}
\maketitle

\vspace{5mm}

\begin{abstract}
A cograph is a simple graph which contains no path on 4 vertices as an induced subgraph.
The vicinal preorder on the vertex set of a graph is defined in terms of inclusions among the neighborhoods of vertices.
The minimum number of chains with respect to the vicinal preorder required to cover the vertex set of a graph $G$ is called the  Dilworth number of $G$.
We prove that for any cograph $G$, the multiplicity of any eigenvalue $\la\ne0,-1$,
does not exceed the Dilworth number of $G$ and  show that this bound is tight.
G. F. Royle [The rank of a cograph,   Electron. J. Combin. 10 (2003), Note 11] proved that if a cograph $G$ has no pair of vertices with the same neighborhood, then $G$ has no 0 eigenvalue, and asked if beside cographs,  there are any other natural classes of graphs for which this property holds. We give a partial answer to this question by showing that an $H$-free family of graphs has this property  if and only if it is a subclass of the family of cographs.
 A similar result is also shown to hold for the $-1$ eigenvalue.

\vspace{5mm}
\noindent {\bf AMS Mathematics Subject Classification\,(2010):}   05C50, 05C75
\end{abstract}

\vspace{5mm}

\section{Introduction}

A cograph is a simple graph which contains no path on four vertices as an induced subgraph.
The family of cographs is the
smallest class of graphs that includes the single-vertex graph and is closed under
complementation and disjoint union. This property justifies the name
`cograph' standing for `complement reducible graph' which was coined in \cite{clb}.
However, this family of graphs was initially defined under different names \cite{k1,k2,j,l,se,su} and since then has been intensively studied.
It is well known that any cograph has a canonical tree representation, called a cotree.
This tree decomposition scheme of cographs is a particular case of the modular decomposition \cite{g} that applies to arbitrary graphs.
Partly because of this property, cographs are interesting from the algorithmic point of view (see \cite[p.~175]{bls}).
As pointed out in \cite{sa}, cographs have numerous applications in areas like parallel computing \cite{noz} or even biology \cite{gkbc} since they can be used to model series-parallel decompositions. For an account on properties of cographs see \cite{bls}.

Cographs have also been studied from an algebraic point of view. Based on a computer search, Sillke \cite{si} conjectured that
 the rank of the adjacency matrix of any cograph is equal to the number of distinct
non-zero rows in this matrix. The conjecture was proved by Royle \cite{roy}. Since then alternative proofs and extensions of this result have appeared \cite{bss,chy,htw,sa}. Furthermore, in \cite{jtt1} an algorithm is introduced for locating eigenvalues of cographs in a given interval.
   In \cite{gh}, we present a new characterization of cographs; namely a graph $G$ is a cograph if and only if no induced subgraph of $G$ has an eigenvalue in the interval $(-1,0)$.  In \cite{gh}, it is also shown that  the multiplicity of any eigenvalue of a cograph $G$ does not exceed
the total number of duplication and coduplication classes of $G$ (see Section~\ref{pre} for definitions) which is not greater that
the sum of the multiplicities of $0$ and $-1$ as eigenvalues of $G$.

In this paper we explore further properties of the eigenvalues  (of the adjacency matrix) of a cograph.
 We consider a relation on the vertex set of a graph $G$ as follows. We define $u\prec v$ if the open neighborhood of $u$ is contained in the closed neighborhood of $v$.
It turns out that `$\prec$' is a preorder (that is reflexive and transitive) which is called the vicinal preorder.
The minimum number of chains with respect to the vicinal preorder required to cover the vertex set of a graph $G$ is called the  Dilworth number of $G$.
In Section~3, we prove that for any cograph $G$, the multiplicity of any eigenvalue $\la\ne0,-1$,
does not exceed the Dilworth number of $G$ and  show that this bound is best possible. This was first conjectured in \cite{gh}.
 In \cite{roy}, Royle proved that if a cograph $G$ has no duplications, then $G$ has no 0 eigenvalue, and asked if beside cographs,  there are any other natural classes of graphs for which this property holds. In Section~4, we give a partial answer to this question by showing that an $H$-free family of graphs $\f$ has this property  if and only if $\f$ is a subclass of the family of cographs.
 A similar result is also shown to hold for the $-1$ eigenvalue. It is also observed that these results can be stated in terms of
 the existence of a basis consisting of weight 2 vectors  for the eigenspace of either $0$ or $-1$ eigenvalues.

\section{Preliminaries}\label{pre}

In this section we introduce the notations and recall  a basic fact  which will be used frequently.
The graphs we consider are all simple and undirected.
For a  graph $G$, we denote  by $V(G)$ the vertex set of $G$.
For two vertices $u,v$, by $u\sim v$ we mean $u$ and $v$ are adjacent.
 If  $V(G)=\{v_1, \ldots , v_n\}$, then the {\em adjacency matrix} of $G$ is an $n \times  n$
 matrix $A(G)$ whose $(i, j)$-entry is $1$ if $v_i\sim v_j$ and  $0$ otherwise.
 By {\em eigenvalues} and {\em rank} of $G$ we mean those of $A(G)$.
 The multiplicity of an eigenvalue $\la$ of $G$ is denoted by $\mul(\la,G)$.
For a vertex $v$ of $G$, let $N_G(v)$ denote the {\em open neighborhood} of $v$, i.e.   the set of
vertices of $G$ adjacent to $v$ and $N_G[v]=N_G(v)\cup\{v\}$ denote the {\em closed neighborhood} of $v$; we will drop
the subscript $G$  when it is clear from the context.
Two vertices $u$ and $v$ of $G$ are called {\em duplicates} if $N(u)=N(v)$
and called {\em coduplicates} if $N[u]=N[v]$. Note that duplicate vertices cannot be
adjacent while coduplicate vertices must be adjacent. We write $u\equiv v$ if $u$ and $v$ are either duplicates or coduplicates.
A subset $S$ of  $V(G)$   such that $N(u)=N(v)$
for any $u, v\in  S$   is called  a  {\em duplication  class} of $G$.  Coduplication  classes  are defined analogously. If $X\subset V(G)$, we use the notation $G-X$ to mean the subgraph of $G$ induced by $V(G)\setminus X$.

An important subclass of cographs are {\em threshold graphs}. These are the graphs which are both a cograph and a split graph (i.e. their vertex sets can be partitioned into a clique and a coclique). For more information see \cite{bls,mp}.

\begin{remark}\label{rem} ({\em Sum rule}) Let $\x$ be an eigenvector for eigenvalue $\la$ of a graph $G$. Then the entries of $\x$ satisfy the following equalities:
\begin{equation}\label{sumrule}
\la\x(v)=\sum_{u\sim v}\x(u),~~\hbox{for all}~v\in V(G).
\end{equation}
From this it is seen that if $\la\ne0$, then $\x$ is constant on each duplication class and if $\la\ne-1$, then $\x$ is constant on each coduplication class.
\end{remark}

\section{Dilworth number and multiplicity of eigenvalues}

In this section we recall a preorder on the vertex set of a graph which is defined
in terms of open/closed neighborhoods of vertices.
In \cite{gh}, it was conjectured that the multiplicity of eigenvalues of any cograph except for $0,-1$ is bounded above by the maximum size of an antichain with respect to this preorder.  We prove this conjecture in this section.

Let $G$ be a graph and consider the following relation on $V(G)$:
 $$u\prec v~~\hbox{if and only if}~\left\{\begin{array}{ll}N[u]\subseteq N[v]&\hbox{if $u\sim v$,}\\N(u)\subseteq N(v)&\hbox{if $u\not\sim v$,}\end{array}\right.$$
or equivalently
$$u\prec v~~\hbox{if and only if}~~N(u)\subseteq N[v].$$
It is easily verified that `$\prec$' is a {\em preorder}  that is reflexive and transitive \cite{fh}.
This preorder is called {\em vicinal preorder}. Note that  `$\prec$' is not antisymmetric, since $u\prec v$ and $v\prec u$ imply only $u\equiv v$.

We consider the chains and antichains in $G$ with respect to the vicinal preorder.
The minimum number of chains with respect to the vicinal preorder required to cover $V(G)$ is called the {\em  Dilworth number} of $G$ and denoted by $\nabla(G)$.
This parameter was first introduced in \cite{fh} (see also \cite{bls}). Also we refer to \cite[Chapter 9]{mp} for several interesting results on Dilworth number and its connection with other graph theoretical concepts.
Note that by Dilworth's theorem, $\nabla(G)$ is equal to  the maximum size of an antichain of $V(G)$ with respect to the vicinal preorder.

\begin{remark}\label{thrstruc} ({\em Structure of threshold graphs}) As it was observed in \cite{ma} (see also \cite{as,ha}), the vertices of any threshold graph $G$ can be partitioned into $t$ non-empty coduplication classes $V_1,\ldots, V_t$ and $t$ non-empty duplication classes $U_1,\ldots, U_t$  such that the vertices in $V_1\cup\cdots\cup V_t$ form a clique and
$$N(u)=V_1\cup\cdots\cup V_i~~\hbox{for any}~ u\in U_i,~1\le i\le t.$$
For an illustration of this structure with $t=5$, see Figure~\ref{figthr}.
\end{remark}

From the structure of threshold graphs it is clear that if $G$ is a threshold graph, then $\nabla(G)=1$. In \cite{ch,fh}, it was observed that the converse is also true.
We give its simple argument here.
If $\nabla(G)=1$, then all the vertices of $G$ form a chain
$v_1\prec v_2\prec\cdots\prec v_n$. First note that $v_i\sim v_{i+1}\not\sim v_{i+2}$ is impossible for any $i$;
since otherwise $N[v_i]\subseteq N[v_{i+1}]$ and $N(v_{i+1})\subseteq N(v_{i+2})$.
Hence, $v_i\in N(v_{i+1})\subset N(v_{i+2})$, and thus $v_{i+2}\in N(v_i)\subset N[v_{i+1}]$ which means $v_{i+1}\sim v_{i+2}$, a contradiction.
 So  there must exist some $j$ such that $v_1\not\sim\cdots\not\sim v_j\sim\cdots\sim v_n$.
It turns out that $G$ is a split graph as the vertices $v_1,\ldots,v_j$ form a coclique  and
  the vertices $v_{j+1},\ldots,v_n$ form a clique.
  Note that any induced subgraph of $G$ has Dilworth number 1. As $\nabla(P_4)=2$, $G$ has no induced subgraph $P_4$, so $G$ is also a cograph which in turn implies that $G$ is a threshold graph.

\begin{figure}\centering
\begin{tikzpicture}
\draw [line width=1.75pt] (-4,-3)-- (-4,-7);
\draw [line width=1.75pt] (-4,-3)-- (-3.4,-4.19);
\draw [line width=1.75pt] (-3.4,-4.19)-- (-3.38,-5.9);
\draw [line width=1.75pt] (-3.38,-5.9)-- (-4,-7);
\draw [line width=1.75pt] (-3.4,-4.19)-- (-4,-7);
\draw [line width=1.75pt] (-4,-3)-- (-4.28,-4.91);
\draw [line width=1.75pt] (-4,-3)-- (-3.38,-5.9);
\draw [line width=1.75pt] (-4.28,-4.91)-- (-4,-7);
\draw [line width=1.75pt] (-4.28,-4.91)-- (-3.4,-4.19);
\draw [line width=1.75pt] (-4.28,-4.91)-- (-3.38,-5.9);
\draw [line width=1.75pt] (-7,-3)-- (-4,-3);
\draw [line width=1.75pt] (-7,-4)-- (-4,-3);
\draw [line width=1.75pt] (-7,-4)-- (-3.4,-4.19);
\draw [line width=1.75pt] (-7,-5)-- (-4,-3);
\draw [line width=1.75pt] (-7,-5)-- (-3.4,-4.19);
\draw [line width=1.75pt] (-7,-5)-- (-4.28,-4.91);
\draw [line width=1.75pt] (-7,-6)-- (-4,-3);
\draw [line width=1.75pt] (-7,-6) to [out=35,in=-160] (-3.4,-4.19);
\draw [line width=1.75pt] (-7,-6)-- (-4.28,-4.91);
\draw [line width=1.75pt] (-7,-6)-- (-3.38,-5.9);
\draw [line width=1.75pt] (-7,-7)-- (-4,-3);
\draw [line width=1.75pt] (-7,-7)-- (-4.28,-4.91);
\draw [line width=1.75pt] (-7,-7)-- (-3.38,-5.9);
\draw [line width=1.75pt] (-7,-7)-- (-4,-7);
\draw[line width=1.75pt] (-7,-7) to [out=30,in=-120] (-3.4,-4.19);
\draw [fill=black] (-4,-3) circle (5.5pt);
\draw[color=black] (-3.6,-2.8) node {$V_1$};
\draw [fill=black] (-3.4,-4.19) circle (5.5pt);
\draw[color=black] (-3.2,-3.8) node {$V_2$};
\draw [fill=black] (-4.28,-4.91) circle (5.5pt);
\draw[color=black] (-4.5,-5.3) node {$V_3$};
\draw [fill=black] (-3.38,-5.9) circle (5.5pt);
\draw[color=black] (-2.9,-6.1) node {$V_4$};
\draw [fill=black] (-4,-7) circle (5.5pt);
\draw[color=black] (-3.5,-7.1) node {$V_5$};
\draw(-7,-4) circle (5.5pt);
\draw(-7,-3) circle (5.5pt);
\draw (-7,-5) circle (5.5pt);
\draw (-7,-6) circle (5.5pt);
\draw (-7,-7) circle (5.5pt);
\draw (-7.5,-3) node {$U_1$};
\draw (-7.5,-4) node {$U_2$};
\draw (-7.5,-5) node {$U_3$};
\draw (-7.5,-6) node {$U_4$};
\draw (-7.5,-7) node {$U_5$};
\end{tikzpicture}
\caption{A threshold graph: $V_i$'s are cliques, $U_i$'s are cocliques, each thick line indicates the edge set of a complete bipartite subgraph on some $U_i,V_j$ }\label{figthr}
\end{figure}
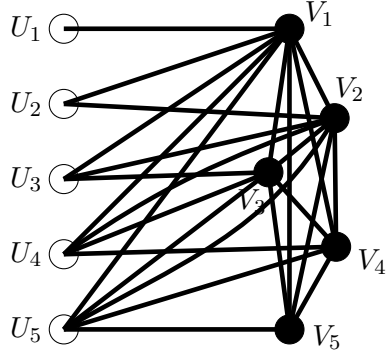

The next result shows that when vicinal preorder is a total order, i.e. the whole  $V(G)$ itself is a chain,
  a strong constraint is imposed on the multiplicities of the  eigenvalues. This result was first proved in \cite{jtt0} (see \cite{gh} for a simpler proof).

\begin{lem}\label{thr} {\rm(\cite{jtt0})} Let $G$ be a threshold graph. Then any eigenvalue $\la\ne0,-1$ is simple.
\end{lem}

Motivated by this result, we \cite{gh} investigated  the connection between eigenvalue multiplicity and  Dilworth number of cographs (as an extension of threshold graphs).
As the vertices of any graph $G$ can be partitioned into $\nabla(G)$ chains and so into $\nabla(G)$ threshold subgraphs, we conjectured \cite{gh} that the Dilworth number is an upper bound for the multiplicity of any eigenvalues $\la\ne0,-1$ in cographs. We prove this conjecture in the next theorem. Before that we present the following crucial lemma on the structure of cographs.

\begin{lem}\label{partition}
Let $G$ be a cograph with Dilworth number $k\ge2$. Then there exists a vertex-partition of $G$ into $k$ threshold graphs such that one of the threshold graphs $H$ of the partition  has the  property that all the vertices of $H$ have the same neighborhood in $G-V(H)$.
\end{lem}
\begin{proof}{Note that  $N_G(u)\subseteq N_G[v]$ if and only if $N_{\overline G}(v)\subseteq N_{\overline G}[u]$. It follows if we have $u_1\prec u_2\prec\cdots\prec u_\ell$ in $G$, then we have $u_\ell\prec\cdots\prec u_2\prec u_1$ in $\overline G$. This in particular implies that
$\nabla(G)=\nabla(\overline G)$ and that if the assertion holds for $G$, then it also holds for $\overline G$. Since connected cographs have disconnected complements (because any cograph is either a union or join of two smaller cographs), we may assume that $G$ is disconnected. Also we may suppose that $G$ has no isolated vertices.

We now proceed by induction on $k\ge2$. First let $k=2$. As $G$ is disconnected, it is a union of two threshold graphs  for which the assertion
trivially holds. Now, let $k\ge3$. If $G$ has a connected component with Dilworth number at least 2, then we are done by the induction hypothesis.
Otherwise, all connected components of $G$ are threshold graphs for which the assertion holds as well.
}\end{proof}

We are now in a position to prove the main result of the paper.

\begin{thm}\label{main}
For any cograph $G$, the multiplicity of any eigenvalue $\la\ne0,-1$,
does not exceed the Dilworth number of $G$. Moreover, there are an infinite family of cographs for which this bound is tight.
\end{thm}
\begin{proof}{Let $G$ be a cograph, $\nabla(G)=k$ and $\la\ne0,-1$ be an eigenvalue of $G$.
We proceed by induction on $k$. If $k=1$, then $G$ is a threshold graph and the assertion follows from Lemma~\ref{thr}.
Let $k\ge2$, and the theorem hold for graphs with Dilworth number at most $k-1$.
If $G$ is disconnected, we are done by the induction hypothesis. So  we may assume that $G$ is connected.
For a contradiction, assume that $\mul(\la,G)\ge k+1$.
By Lemma~\ref{partition}, $G$ contains a threshold subgraph $H'$ such that any two vertices of $H'$ have the same neighborhood in $G-V(H')$ with
$G-V(H')$ having Dilworth number $k-1$.
 Then $H'=H\cup I$ where $I$ (possibly empty) is the subgraph consisting of isolated vertices of $H'$ and $H$ is a connected threshold graph with at least one edge.
  Let $V_1,\ldots,V_t$ and $U_1,\ldots,U_t$ be the partition of $V(H)$ according to Remark~\ref{thrstruc}.
  Let $u_t\in U_t$.
 There is a $k$-dimensional subspace $\Omega$ of eigenvectors of $G$ for $\la$ which vanishes on $u_t$.
 Let $\x\in\Omega$. Note that $N_H(u_t)=V_1\cup\cdots\cup V_t$. By the sum rule and since $\x(u_t)=0$,
 $$0=\la\x(u_t)=\sum_{v\in V_1\cup\cdots\cup V_t}\x(v)+ \alpha,$$
 where $\alpha$ is the sum of entries of $\x$ on the neighbors of $u_t$ outside $H'$; by the way $H'$ is chosen, this is constant for all the vertices of $H'$.
 Let $v_t\in V_t$.
 Then $$N_H(v_t)=U_t\cup V_1\cup\cdots\cup V_t\setminus\{v_t\}.$$
 By Remark~\ref{rem}, for all vertices $u\in U_t$ we have $u\equiv u_t$, and thus $\x(u)=\x(u_t)=0$. It follows that

  \begin{align*}
\la\x(v_t)&=\sum_{v\in N_H(v_t)}\x(v)+ \alpha\\
&=-\x(v_t)+\sum_{v\in V_1\cup\cdots\cup V_t}\x(v)+ \alpha\\&=-\x(v_t)+\la\x(u_t)\\&=-\x(v_t).
\end{align*}
  As $\la\ne-1$, it follows that $\x(v_t)=0$. Again by Remark~\ref{rem}, we have $\x(v)=0$ for all $v\in V_t$.
  Now let $u_{t-1}\in U_{t-1}$. We have $N_H(u_{t-1})= V_1\cup\cdots\cup V_{t-1}$, and since $\x={\bf0}$ on $V_t$,
  $$\la\x(u_{t-1})=\sum_{v\in V_1\cup\cdots\cup V_{t-1}}\x(v)+\alpha=\sum_{v\in V_1\cup\cdots\cup V_t}\x(v)+\alpha=\la\x(u_t)=0.$$
 Again, it follows that $\x={\bf0}$ on $U_{t-1}$.
  Let $v_{t-1}\in V_{t-1}$.
 Then $$N_H(v_{t-1})=U_t\cup U_{t-1}\cup V_1\cup\cdots\cup V_t\setminus\{v_{t-1}\}.$$
 Therefore,
 \begin{align*}
\la\x(v_{t-1})&=-\x(v_{t-1})+\sum_{v\in U_t\cup U_{t-1}\cup V_1\cup\cdots\cup V_t}\x(v)+ \alpha\\
&=  -\x(v_{t-1})+\sum_{v\in V_1\cup\cdots\cup V_{t-1}}\x(v)+ \alpha\\
&=-\x(v_{t-1})+\la\x(u_{t-1})\\&=-\x(v_{t-1}).
\end{align*}
 As $\la\ne-1$, it follows that $\x(v_{t-1})=0$. By Remark~\ref{rem}, we have $\x(v)=0$ for all $v\in V_{t-1}$.

Continuing this procedure, we alternately choose vertices $u_{t-2},v_{t-2},\ldots,u_1,v_1$ where similar to the above argument we see that $\x$ vanishes on $U_{t-2},V_{t-2},\ldots,U_1,V_1$. So $\x=\bf0$ on the whole $V(H)$ and we must have $\alpha=0$ which in turn implies that $\x$ is zero on $I$. Hence we conclude that $\x$ is zero on $V(H')$. It turns out that if for any $\x\in\Omega$ we remove the entries corresponding to the vertices of $H'$  from $\x$, the resulting vector is an eigenvector of $G-V(H')$ for eigenvalue $\la$. This means that the graph $G-V(H')$ with Dilworth number $k-1$ has the eigenvalue $\la$ with multiplicity $k$, which is a contradiction. Hence the assertion follows.

Finally, we present an infinite family of graphs for which the multiplicity of an eigenvalue $\la\ne0,-1$ is equal to the Dilworth number.
It is easy to construct disconnected cographs for which the equality holds (just take $k$ copies of a fixed threshold graph).
Here we show that the bound can be achieved by connected cographs.
Let $s,k\ge1$ and $G=G(s,k):=K_1\vee(K_{s,\ldots,s}\cup(s^2-s)K_1)$, in which $K_{s,\ldots,s}$ is a complete $k$-partite graph with parts of size $s$, and `$\vee$' denotes the join of two graphs.
Let $v$ be the vertex of $G$ adjacent to all the other vertices, $B_1,\ldots,B_k$ be the parts of the complete $k$-partite subgraph, and $C$ be the set of $s^2-s$ pendant vertices.
It is easily seen that $B_1,\ldots,B_k,\{v\},C$ make an equitable partition\footnote{For the definition and properties of equitable partitions, we refer to \cite[p.~24]{bh}.} of $G$, with the quotient matrix
  $$Q=\left(
    \begin{array}{ccc|cc}
       &  & &   1 & 0 \\
       & s(J_k-I_k)  & & \vdots & \vdots \\
         & &  & 1 & 0 \\ \hline
        s & \cdots & s & 0 & s^2-s \\
      0 & \cdots & 0 & 1&0 \\
    \end{array}
  \right),$$
  where $J_k$ is the $k\times k$ all 1's matrix.
Now, we have
$$Q+sI=\left(
    \begin{array}{ccc|cc}
      s & \cdots & s & 1 & 0 \\
      \vdots &   &\vdots & \vdots & \vdots \\
        s & \cdots & s & 1 & 0 \\ \hline
        s & \cdots & s & s & s^2-s \\
      0 & \cdots & 0 & 1&s \\
    \end{array}
  \right).$$
It is seen that all the rows of $Q+sI$ can be obtained by linear combinations of the last two rows. So ${\rm rank}(Q+sI)=2$ which means $\la=-s$ is an eigenvalue of $Q$ with multiplicity $k$ and thus an eigenvalue of $G$ with multiplicity at least $k$.
Also the vertices of $G$ can be partitioned into $k$ chains with respect to vicinal preorder, namely  $C\cup B_1\cup\{v\}$, $B_2,\ldots, B_k$. Hence $\nabla(G)\le k$.
}\end{proof}

\begin{remark} Theorem~\ref{main} cannot hold for general graphs. Here we describe a family of counterexamples.
For any positive integer $n$, the {\em cocktail party graph} $\cp(n)$ is the graph obtained from the complete graph $K_{2n}$ by removing a perfect
matching. In fact, $\cp(n)$ is the complete $n$-partite graph $K_{2,\ldots,2}$. Let $H$ be a graph with vertices $v_1, \ldots , v_n$,
and $a_1,\ldots , a_n$ be non-negative integers. The {\em generalized line graph}
$L(H; a_1,\ldots , a_n)$ consists of disjoint copies of $L(H)$  and $\cp(a_1), \ldots ,\cp(a_n)$
together with all edges joining a vertex $\{v_i , v_j\}$ of $L(H)$ with each vertex in $\cp(a_i)$ and $\cp(a_j)$.
It is known that all the eigenvalues of generalized line graphs are greater than or equal to $-2$. Moreover, by Theorem 2.2.8 of \cite{crs}, if $H$ has $m$ edges and if not all $a_i$'s are zero, then
\begin{equation}\label{mult-2}
\mul(-2,L(H; a_1,\ldots , a_n))=m-n+\sum_{i=1}^na_i.
\end{equation}
Consider the generalized line graph $G=G(k):=L(K_{1,k};k,1,\ldots,1)$. In fact $G$ is obtained from the graph $K_k\vee\cp(k)$ by attaching two pendant vertices
to each of the vertices of $K_k$. Let $u_1,\ldots,u_k$ be the vertices of $K_k$, $U_1,\ldots,U_k$ be the parts of $\cp(k)$, and $w_{i1},w_{i2}$ be the pendant vertices attached to $u_i$, for $i=1,\ldots,k$. Then for each $i$, $\{u_i,w_{i1},w_{i2}\}\cup U_i$ form a chain, so $\nabla(G)\le k$. On the other hand, $U_1,\ldots,U_k$ is an antichain in the vicinal preorder of $G$, so $\nabla(G)\ge k$. Thus we have $\nabla(G)=k$.
However by \eqref{mult-2}, $\mul(-2,G)=2k-1$.
\end{remark}

\begin{remark} The bound given in Theorem~\ref{main}  can be arbitrarily loose. To see this, let $r_1,\ldots,r_k$ be distinct integers greater than $1$ and $G=K_{r_1,\ldots,r_k}$. Clearly, $G$ is a cograph with $\nabla(G)=k$. However, all the non-zero eigenvalues of $G$ are simple (see  \cite[Theorem~1]{eh}).
\end{remark}

\section{$H$-free graphs that only (co)duplications reduce their rank}

In \cite{gh} a characterization of cographs based on graph eigenvalues was given;  namely a graph $G$ is a cograph if and only if no induced subgraph of $G$ has an eigenvalue in the interval $(-1,0)$. In this section we give another characterization for the family of cographs
which is based on the presence of eigenvalue  $0$ or $-1$ in connection with the existence of duplications or coduplications.

Regarding the  presence of  eigenvalue $0$ in cographs, Royle \cite{roy} proved the following result confirming a conjecture by  Sillke \cite{si}.
\begin{lem}\label{roy0} If a cograph $G$ has no duplications, then $A(G)$ has full rank.
\end{lem}

The following result concerning $-1$ eigenvalues is also implicit in \cite{roy}  (see also \cite{bls,chy,sa}).

\begin{lem}\label{roy-1} If a cograph $G$ has no coduplications, then $A(G)+I$ has full rank.
\end{lem}

We say that a graph $G$ satisfies  {\em Duplication-Rank Property} (DRP for short) if the following holds:
$$\hbox{$G$ has a duplication, or $A(G)$ has full rank.}$$
Similarly, we say that $G$ satisfies {\em Coduplication-Rank Property} (CDRP for short) if:
$$\hbox{$G$ has a coduplication, or $A(G)+I$ has full rank.}$$

Lemma~\ref{roy0} says that cographs satisfy DRP. Motivated by this, Royle~\cite{roy} posed the following question:
$$\hbox{\em Beside cographs, are there any other natural classes of graphs for which DRP holds?}$$

For a given graph $H$, the family of $H$-free graphs is the set of all graphs which do not contain $H$ as an induced subgraph.
A `natural' class of graphs is the family of $H$-free graphs for a specific graph $H$. Here we give the answer to the Royle's question for $H$-free families of graphs.
The same result is given for graphs with CDRP.

For a graph $H$, we denote the family of  $H$-free graphs by $\f(H)$.
We say that $\f(H)$ satisfies DRP (or CDRP) if any $G\in\f(H)$ does.
Also by $H\sg G$, we mean that $H$ is an induced subgraph of $G$.

The following result shows that if an $H$-free family of graphs satisfies DRP or CDRP it must be contained in the family of cographs.

\begin{thm}\label{RoyThm} Let $\f(H)$ be the family of $H$-free graphs.
\begin{itemize}
\item[\rm(i)]  $\f(H)$ satisfies DRP if and only if $H$ is an induced subgraph of $P_4$.
\item[\rm(ii)]  $\f(H)$ satisfies CDRP if and only if $H$ is an induced subgraph of $P_4$.
\end{itemize}

\end{thm}
\begin{proof}{(i) If $H\sg P_4$, then $\f(H)\subseteq\f(P_4)$, and so by Lemma~\ref{roy0}, $\f(H)$ satisfies DRP, showing the `sufficiency'.
 For the `necessity,' assume that for a given graph $H$,
$\f(H)$ satisfies DRP. We show that $H\sg P_4$.

If $H$ has one or two vertices, then $H\sg P_4$, and we are done.
Let $H$ have three vertices. If $H=P_3$ or $H=P_2\cup P_1$, then $H\sg P_4$. We show that it is impossible that $H=K_3$ or $\overline{K_3}$.
Note that $P_5\in\f(K_3)$ and it does not satisfy DRP as it has no duplications but has a 0 eigenvalue (its $0$-eigenvector is illustrated in Figure~\ref{figP5}).
The graph depicted in Figure~\ref{figDRP} belongs to $\f(\overline{K_3})$ but it does not satisfies DRP; it has no duplications and has a 0 eigenvalue (its $0$-eigenvector indicated in the picture).
Hence the assertion follows for three-vertex graphs.
\begin{figure}\centering
\begin{tikzpicture}
\draw [line width=.5pt] (-2,0)-- (2,0);
\draw [fill=black] (0,0) circle (2.5pt);
\draw (0,0.3) node {$-1$};
\draw (0,-0.3) node {$0$};
\draw [fill=black] (-1,0) circle (2.5pt);
\draw (-1,0.3) node {$0$};
\draw (-1,-0.3) node {$-1$};
\draw [fill=black] (-2,0) circle (2.5pt);
\draw(-2,0.3) node {$1$};
\draw(-2,-0.3) node {$1$};
\draw [fill=black] (1,0) circle (2.5pt);
\draw (1,0.3) node {$0$};
\draw (1,-0.3) node {$-1$};
\draw [fill=black] (2,0) circle (2.5pt);
\draw (2,0.3) node {$1$};
\draw (2,-0.3) node {$1$};
\end{tikzpicture}
\caption{ $P_5$ with its $0$-eigenvector shown on the above and $-1$-eigenvector on the below}
\label{figP5} \end{figure}
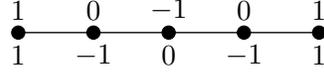

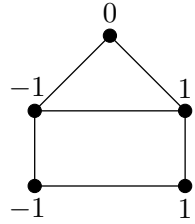
\begin{figure}\centering
\begin{tikzpicture}
\draw [line width=.5pt] (-1,1)-- (1,1);
\draw [line width=.5pt] (0,2)-- (1,1);
\draw [line width=.5pt] (0,2)-- (-1,1);
\draw [line width=.5pt] (-1,1)-- (-1,0);
\draw [line width=.5pt] (1,0)-- (-1,0);
\draw [line width=.5pt] (1,1)-- (1,0);
\draw [fill=black] (-1,0) circle (2.5pt);
\draw  (-1.1,-0.3) node {$-1$};
\draw [fill=black] (-1,1) circle (2.5pt);
\draw  (-1.1,1.3) node {$-1$};
\draw [fill=black] (1,1) circle (2.5pt);
\draw  (1,1.3) node {$1$};
\draw [fill=black] (1,0) circle (2.5pt);
\draw  (1,-0.3) node {$1$};
\draw [fill=black] (0,2) circle (2.5pt);
\draw  (0,2.3) node {$0$};
\end{tikzpicture}
\caption{ A graph in $\f(\overline{K_3})\cap\f(P_5)\cap\f(2K_2)$ and its $0$-eigenvector }
\label{figDRP} \end{figure}

Suppose that $H$ has four vertices. If $H=P_4$, there is nothing to prove. If $K_3\sg H$ or $\overline{K_3}\sg H$, then $\f(K_3)\subseteq\f(H)$ or
 $\f(\overline{K_3})\subseteq\f(H)$ and so $\f(H)$ does not satisfies DRP.
It remains to consider 4-vertex bipartite graphs  with no 3-coclique. Besides $P_4$, there are only two such graphs, namely $K_{2,2}$ and $2K_2$.
But $P_5$ is a $K_{2,2}$-free graph not satisfying DRP and the graph of Figure~\ref{figDRP} is a $2K_2$-free graph not satisfying DRP. Therefore, the assertion holds for four-vertex graphs.

Now let $H$ have five or more vertices. Let $H'$ be a five-vertex graph with $H'\sg H$.
As $\f(H')\subseteq\f(H)$, the family $\f(H')$ also satisfies DRP.
By the argument for four-vertex graphs, $P_4$ is the only four-vertex graph with $\f(P_4)$ satisfying DRP.
It follows that all four-vertex induced subgraphs of $H'$ must be isomorphic to $P_4$.
There is a unique graph $H'$ with this property, namely the 5-cycle $C_5$.
However, $P_5\in\f(C_5)$ and it does not satisfy DRP. It turns out that for no graph $H$ with five or more vertices, $\f(H)$ satisfies DRP. This completes the proof.

(ii) Similar to the proof of (i), the `sufficiency' follows from Lemma~\ref{roy-1}.  For the `necessity,' assume that
$\f(H)$ satisfies CDRP. There is nothing to prove if $H$ has one or two vertices. Let $H$ have three vertices. If $H=P_3$ or $H=P_2\cup P_1$, then $H\sg P_4$.
Note that $P_5$ does not satisfies CDRP (as shown in Figure~\ref{figP5}) but belongs to $\f(K_3)$ and the graph of Figure~\ref{figCDRP} does not satisfies CDRP but belongs to $\f(\overline{K_3})$.
Hence the assertion follows for three-vertex graphs.
\begin{figure}\centering
\begin{tikzpicture}
\draw [line width=.5pt] (0,1)-- (0,-1);
\draw [line width=.5pt] (-3,0)-- (3,0);
\draw [line width=.5pt] (0,1)-- (-3,0);
\draw [line width=.5pt] (-3,0)-- (0,-1);
\draw [line width=.5pt] (3,0)-- (0,-1);
\draw [line width=.5pt] (0,1)-- (3,0);
\draw [line width=.5pt] (0,1)-- (1,0);
\draw [line width=.5pt] (1,0)-- (0,-1);
\draw [line width=.5pt] (0,-1)-- (-1,0);
\draw [fill=black] (0,1) circle (2.5pt);
\draw (0,1.35) node {$1$};
\draw [fill=black] (0,-1) circle (2.5pt);
\draw (0,-1.3) node {$-1$};
\draw [fill=black] (1,0) circle (2.5pt);
\draw (1,0.3) node {$1$};
\draw [fill=black] (-1,0) circle (2.5pt);
\draw (-1,0.3) node {$0$};
\draw [fill=black] (3,0) circle (2.5pt);
\draw (3,0.3) node {$-1$};
\draw [fill=black] (-3,0) circle (2.5pt);
\draw (-3,0.3) node {$0$};
\end{tikzpicture}
\caption{A graph in $\f(\overline{K_3})\cap\f(2K_2)$ and its $-1$-eigenvector}\label{figCDRP}
\end{figure}
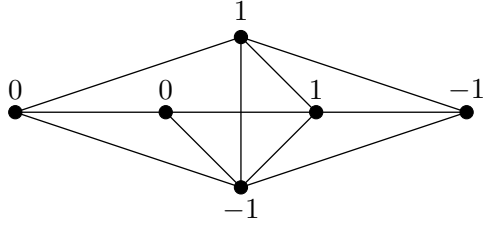

Suppose that $H$ has four vertices. If $H=P_4$, there is nothing to prove. If $K_3\sg H$ or $\overline{K_3}\sg H$, then $\f(K_3)\subseteq\f(H)$ or
 $\f(\overline{K_3})\subseteq\f(H)$ and so $\f(H)$ does not satisfies CDRP.
It remains to consider $H=K_{2,2}$ and $2K_2$.
But $P_5$ is a $K_{2,2}$-free graph not satisfying CDRP and the graph of Figure~\ref{figCDRP} is a $2K_2$-free graph not satisfying CDRP. Therefore, the assertion holds for four-vertex graphs.

Now let $H$ have five or more vertices. Let $H'$ be a five-vertex graph with $H'\sg H$.
As $\f(H')\subseteq\f(H)$, the family $\f(H')$ also satisfies CDRP.
By the argument for four-vertex graphs, we have that all four-vertex induced subgraphs of $H'$ must be isomorphic to $P_4$ and so $H'=C_5$.
However, $P_5\in\f(C_5)$ and it does not satisfy CDRP. It turns out that for no graph $H$ with five or more vertices, $\f(H)$ satisfies CDRP,  completing the proof.
}\end{proof}

Any pair of duplicate vertices $u,v$ in a graph $G$ give rise to a null-vector of $A(G)$ of weight two (the vector whose components corresponding to $u,v$ are $1,-1$,  and zero elsewhere). Conversely, any null-vector $\x$
of $A(G)$ of weight two comes from a pair of duplicate vertices.
To see this, suppose that $x_u$ and $x_v$ are the two non-zero components of $\x$. As $A(G)$ is a $0,1$-matrix, we must have $x_u=-x_v$. It turns out that the rows of  $A(G)$ corresponding to
$u$ and $v$ are identical which means that $u$ and $v$ are duplicates.
Hence, we observe that any null-vector of $A(G)$ of weight two corresponds with a pair of duplicate vertices in $G$.
 Similarly, null-vectors
of $A(G)+I$ of weight two correspond to  pairs of coduplicate vertices.
In \cite{sa} (see also \cite{gh}) it is shown that if $G$ is a cograph, then the null-space of $A(G)$ (resp. $A(G)+I$) has a basis consisting of the weight-two null-vectors  corresponding to duplicate (resp. coduplicate) pairs. Hence the following can be deduced from Theorem~\ref{RoyThm}.

\begin{cor}  Let $\f(H)$ be the family of $H$-free graphs.
\begin{itemize}
\item[\rm(i)] For all graphs $G\in\f(H)$ the null-space of $A(G)$ has a basis consisting of vectors of weight two if and only if $H$ is an induced subgraph of $P_4$.
\item[\rm(ii)] For all graphs $G\in\f(H)$ the null-space of $A(G)+I$ has a basis consisting of vectors of weight two if and only if $H$ is an induced subgraph of $P_4$.
   \end{itemize}
\end{cor}

\section*{Acknowledgments}
The research of the author was in part supported by a grant from IPM (No. 95050114).


\begin{thebibliography}{}
\bibitem{as} M. Andeli\'c and S.K. Simi\'c, Some notes on the threshold graphs, {\em Discrete Math.} {\bf310} (2010), 2241--2248.
 \bibitem{bss}  T. Biyiko\u glu, S.K. Simi\'c, and  Z. Stani\'c,  Some notes on spectra of cographs, {\em Ars Combin.} {\bf100} (2011), 421--434.
\bibitem{bls} A. Brandst\"adt, V.B. Le, and J.P. Spinrad, {\em Graph Classes: A Survey},
 Society for Industrial and Applied Mathematics (SIAM), Philadelphia, PA, 1999.
 \bibitem{bh} A.E. Brouwer and W.H. Haemers, {\em Spectra of Graphs}, Springer, New York, 2012.
\bibitem{chy}  G.J. Chang, L.-H. Huang, and  H.-G. Yeh, On the rank of a cograph, {\em Linear Algebra Appl.} {\bf429} (2008), 601--605.
\bibitem{ch} V. Chv\'atal and P.L. Hammer, Aggregation of inequalities in integer programming, {\em Ann. Discrete Math.} {\bf1} (1977), 145--162.
\bibitem{clb} D.G. Corneil, H. Lerchs, and L.S. Burlingham, Complement reducible graphs, {\em Discrete Appl. Math.} {\bf3} (1981), 163--174.
\bibitem{crs} D. Cvetkovi\'c, P. Rowlinson, and S. Simi\'c, {\em Spectral Generalizations of Line Graphs, On Graphs with Least Eigenvalue $-2$}, London Math. Society Lecture Note Series, Cambridge Univ. Press, Cambridge, 2004.
\bibitem{eh} F. Esser and F. Harary, On the spectrum of a complete multipartite graph, {\em European J. Combin.} {\bf1} (1980), 211--218.
\bibitem{fh} S. F\"oldes and P.L. Hammer, The Dilworth number of a graph, {\em Ann. Discrete Math.} {\bf2} (1978), 211--219.
\bibitem{gkbc} J. Gagneur, R. Krause, T. Bouwmeester, and G. Casari,  Modular
decomposition of protein-protein interaction networks,  {\em Genome Biology} {\bf5} (2004), R57.
\bibitem{g} T. Gallai, Transitiv orientierbarer graphen, {\em Acta Math. Acad. Sci. Hungar} {\bf18} (1967),  25--66.
\bibitem{gh} E. Ghorbani, Spectral properties of cographs and $P_5$-free graphs, {\em Linear Multilinear Algebra}, to appear.
\bibitem{ha} F. Harary, The structure of threshold graphs, {\em Riv. Mat. Sci. Econom. Social.} {\bf2} (1979), 169--172.
\bibitem{htw} L.-H. Huang, B.-S. Tam, and S.-H. Wu,  Graphs whose adjacency matrices have rank equal to the number of distinct nonzero rows, {\em Linear Algebra Appl.} {\bf438} (2013), 4008--4040.
\bibitem{k1} A.K. Kelmans, The number of trees in a graph. I, {\em Autom. Remote Control} {\bf26} (1965), 2118--2129.
\bibitem{k2} A.K. Kelmans, The number of trees in a graph. II, {\em Autom. Remote Control} {\bf27} (1966), 233--241.
\bibitem{jtt0} D.P. Jacobs, V. Trevisan, and F. Tura, Eigenvalue location in threshold graphs, {\em Linear Algebra Appl.} {\bf439} (2013), 2762--2773.
\bibitem{jtt1} D.P. Jacobs, V. Trevisan, and F. Tura, Eigenvalue location in cographs, {\em Discrete Appl. Math.} {\bf245} (2018), 220--235.

\bibitem{j} H.A. Jung,  On a class of posets and the corresponding comparability graphs, {\em J. Combin. Theory Series B} {\bf24} (1978), 125--133.
\bibitem{l} H. Lerchs,  On cliques and kernels, Technical Report, Dept. of Comp. Sci., Univ. of Toronto, 1971.
\bibitem{lz} J. Liu and H.S. Zhou, Dominating subgraphs in graphs with some forbidden structures, {\em Discrete Math.} {\bf135} (1994), 163--168.
\bibitem{mp} N.V.R. Mahadev and U.N. Peled, {\em Threshold Graphs and Related Topics}, Annals of Discrete Mathematics, North–Holland Publishing Co., Amsterdam, 1995.
\bibitem{ma} P.   Manca, On a simple characterisation of threshold graphs, {\em Riv. Mat. Sci. Econom. Social.} {\bf2} (1979), 3--8.
\bibitem{mt} A. Mohammadian and V. Trevisan, Some spectral properties of cographs, {\em Discrete Math.} {\bf339} (2016), 1261--1264.
\bibitem{noz} K. Nakano, S. Olariu,  and A. Zomaya,  A time-optimal solution for the path
cover problem on cographs, {\em Theoret. Comput. Sci.} {\bf290} (2003), 1541--1556.
\bibitem{roy}  G.F. Royle,  The rank of a cograph,  {\em Electron. J. Combin.} {\bf 10} (2003), Note 11.
\bibitem{sa}   T. Sander,  On certain eigenspaces of cographs, {\em  Electron. J. Combin.} {\bf15} (2008), Research Paper 140.
\bibitem{se} D. Seinsche, On a property of the class of $n$-colorable graphs, {\em J. Combin. Theory Series B} {\bf16} (1974),  191--193.
\bibitem{si} T. Sillke, Graphs with maximal rank, {\footnotesize{\tt https://www.math.uni-bielefeld.de/\~{}sillke/PROBLEMS/cograph}}.
\bibitem{su} D.P. Sumner, Dacey graphs, {\em J. Austral. Math. Soc.} {\bf18}  (1974),  492--502.

\end{thebibliography}
\end{document}